\documentclass{amsart}

\usepackage[margin=1.2in]{geometry}
\usepackage{amssymb,enumerate,amsmath,amsxtra,geometry}

\newtheorem*{Lem}{Lemma}

\newtheorem*{Def}{Definition}
\newtheorem*{hypoth}{Hypotheses}
\newtheorem*{Prop}{Proposition}

\newtheorem*{Cor}{Corollary}

\newtheorem*{ThmA}{Theorem A}
\newtheorem*{ThmB}{Theorem B}

\newtheorem*{Thm}{Theorem}

\allowdisplaybreaks

\numberwithin{equation}{subsection}

\theoremstyle{remark}
\newtheorem*{Rmk}{Remark}

\newcommand{\uv}{{\rm U}(V)}
\newcommand{\guv}{{\rm GU}(V)}

\newcommand{\ep}{\epsilon}

\newcommand{\vform}{\langle\mspace{7mu},\mspace{7mu}\rangle}

\newcommand{\pic}{\pi^{\vee}}
\newcommand{\fc}{f^{\vee}}
\newcommand{\thf}{f^{\theta}}

\newcommand{\pii}{\pi^{\iota}}

\newcommand{\ui}{{}^{\iota}}
\newcommand{\ut}{{}^{\theta}}

\newcommand{\cx}{\mathbb{C}}
\newcommand{\tr}{{\rm tr}}

\newcommand{\ev}{\text{End}_E(V)}
\newcommand{\av}{\text{Aut}_E(V)}

\newcommand{\lieg}{\mathfrak{u}(V)}

\newcommand{\lietg}{\mathfrak{gu}(V)}

\newcommand{\lag}{\mathfrak{g}}

\newcommand{\mtl}{\dot{\mathcal{L}}}
\newcommand{\mttl}{\ddot{\mathcal{L}}}
\newcommand{\ml}{\mathcal{L}}

\newcommand{\mlh}{\widehat{\mathcal{L}}}

\newcommand{\eo}{\mathfrak{o}_E}

\newcommand{\fo}{\mathfrak{o}_F}

\newcommand{\fp}{\mathfrak{p}_F}
\newcommand{\fl}{\mathfrak{o}_L}

\begin{document}

\thanks{}

\title[Dualizing involutions for classical and similitude groups]{Dualizing involutions for classical and similitude groups over local non-archimedean fields} 
 
\author[A. Roche]{Alan Roche}
\address{Dept. of Mathematics, University of Oklahoma, Norman OK 73019-3103.}
\email{aroche@math.ou.edu}
\author[C. R. Vinroot]{C.~Ryan Vinroot}
\thanks{Vinroot was supported in part by a grant from the Simons Foundation, Award \#280496}
\address{Dept. of Mathematics, College of William and Mary, P.O. Box 8795,  Williamsburg, VA 23187-8795.}
\email{vinroot@math.wm.edu}


\begin{abstract}
Building on ideas of Tupan, we give an elementary proof of a result of  M{\oe}glin, Vign\'{e}ras and Waldspurger on the 
existence of automorphisms of many $p$-adic classical groups that take each irreducible smooth representation to its dual.  
Our proof also applies to the corresponding similitude groups. 
\end{abstract}

\maketitle 

\section*{Introduction}

Let $F$ be a non-archimedean local field and let $G$ be the group of $F$-points of a reductive $F$-group. 
Let $\iota$ be an automorphism of $G$ of order at most two. We call $\iota$ a {\it dualizing involution} if it takes each irreducible smooth representation of $G$ to its smooth dual or contragredient. 
An early example comes from a paper of Gelfand and Kazhdan \cite{GK}. This shows, via a
geometric method, that transpose-inverse is a dualizing involution of ${\rm GL}_n(F)$. 
By adapting Gelfand and Kazhdan's approach, M{\oe}glin, Vign\'{e}ras and Waldspurger  proved the existence of dualizing involutions for
many classical $p$-adic groups \cite{MoViWa87} Chap.~IV \S~II. 

The  impetus for this paper stems from more recent work of Tupan \cite{Tupan} that   
rederives Gelfand and Kazhdan's result by entirely elementary means. We adapt Tupan's method
so that it applies to the classical groups of \cite{MoViWa87} as well as the corresponding similitude groups. 
The paper \cite{RV} contains another approach to these results via existence of characters  \cite{H-C, Ad-Kor}. 

Let $\fo$ denote the valuation ring of $F$ and fix a uniformizer $\varpi$ in $F$.  
The basis of Tupan's method has two parts: 
(i) the observation that if $\ml$ is an $\fo$-order in ${\rm M}_n(F)$ then the family $\{ 1+ \varpi^k \ml \}_{k \geq 1}$
 forms a neighborhood basis of the identity in ${\rm GL}_n(F)$ consisting of compact open subgroups and 
(ii) a classical result in linear algebra that any square matrix is conjugate to its transpose by a symmetric matrix 
(see, for example, \cite{Kap} page 76).

For the classical and similitude groups $G$ that we consider, Theorem~A of \cite{RV} provides a natural analogue of (ii). More precisely, it gives
an anti-involution $\theta$ of $G$ (i.e., $\ut (ab) = \ut b  \,\ut a$ for all $a, b \in G$ and $\theta^2 = 1$) with the following property: for any $x \in G$ there is a $g \in G$ with $\ut g = g$ such that $g xg^{-1} =  \ut x$.  The dualizing involution $\iota$ of $G$ is 
then given by $\ui g = \ut g^{-1}$ for $g \in G$. In place of (i), we use a suitable $\fo$-lattice $\ml$ in the Lie algebra $\lag$ of $G$. 
For a certain dense open subset $\lag_1$ of $\lag$, there is a Cayley map $c:\lag_1 \to G$ such that $\lag_1$ contains $\varpi \ml$. 
The family $\{ c (\varpi^k \ml) \}_{k \geq 1}$ is then a neighborhood basis of the identity in $G$ that again consists of compact open subgroups.

In this way, our setting and Tupan's fit into a common framework. After a more precise statement of results in \S1, we introduce 
this framework in \S2. It allows us to present an axiomatic version and slight simplification of Tupan's original method. 
We show  in \S4 that our family of groups fits inside the framework. 
This relies on properties of Cayley maps, in particular a Cayley map for similitude groups, that we study in \S3. 
Our use of Cayley maps means that we have to exclude the case of even residual characteristic. 

\section{Preliminaries and statement of results}
\subsection{}
As above, let $F$ be a non-archimedean local field and let $G$ be the $F$-points of a reductive $F$-group. 
Via the topology on $F$, the group $G$ is naturally a locally profinite unimodular topological group.
For any irreducible smooth representation $\pi$ of $G$, we write $\pic$ for the smooth dual or contragredient of $\pi$.  

\begin{Def}
{\rm Let $\iota$ be a  continuous automorphism of $G$ of order at most two. We say that 
$\iota$ is a {\it dualizing involution}
of $G$ if $\pii \cong \pic$ for all irreducible smooth representations $\pi$ of $G$ where $\pii = \pi \circ \iota$.}
\end{Def}

\subsection{}
To introduce the family of classical and similitude groups that we work with, let 
$E/F$ be a field extension with $[E:F] \leq 2$. We assume also that the residual characteristic of $F$ is odd. In particular, 
the characteristic of $F$ cannot be even. This assumption is necessitated by our use of Cayley maps in \S\ref{Cayley maps}.

We write $\tau$ for the generator of ${\rm Gal} (E/F)$.
Thus $\tau$ has order two when $[E:F] = 2$ and  $\tau = 1$ when $E=F$. 
Let $V$ be a finite dimensional vector space over $E$ with a non-degenerate $\ep$-hermitian form $\vform$ for $\ep = \pm 1$. 
We take $\vform$ to be linear in the first variable:
\[
       \langle \alpha u +  \beta v, w \rangle =   \alpha \langle u, w  \rangle + \beta \langle  v,w \rangle \,\,\,\text{and} \,\,\,
          \langle v, w \rangle = \ep \, \tau ( \langle w, v \rangle )          
       \]
for all $\alpha, \beta \in E$ and $ u, v, w \in V$. It follows that $\vform$ is $\tau$-linear in the second variable: 
\[
       \langle u,  \alpha v +  \beta w \rangle =   \tau(\alpha)  \langle u, v  \rangle + \tau(\beta)  \langle  u , w \rangle.
       \]

Let $\uv$ denote the group of isometries of $\vform$ and $\guv$ the corresponding similitude group: 
\begin{align*} 
   \uv  &= \{ g \in {\rm Aut}_E (V) :  \langle g v, g v' \rangle = \langle v, v' \rangle, \,\,\, \forall \,\, v, v' \in V \}, \\
      \guv  &= \{ g \in {\rm Aut}_E (V) :  \langle g v, g v' \rangle =  \beta \langle v, v' \rangle, \, \text{for some scalar $\beta$}, 
 \,  \forall \,\, v, v' \in V \}.
\end{align*} 
For $g \in \guv$ with associated scalar $\beta$ we often write $\mu(g) = \beta$. This is the {\it multiplier} of $g$. 
Note $\beta \in F^\times$. Indeed, applying $\tau$ to  $ \langle g v, g v' \rangle =  \beta \langle v, v' \rangle$ 
gives $\tau(\beta) = \beta$.

\subsection{}
We recall a definition from \cite {RV}.
\begin{Def}
{\rm Let $h \in {\rm Aut}_F(V)$. 
\begin{enumerate}[(1)]
\item
We say that $h$ is {\it anti-unitary} if  
\[
           \langle h v, h v' \rangle    =      \langle   v', v \rangle, \,\, \quad \forall \,\, v, v' \in V.
           \]
\item           
We say also that $h$ is an {\it anti-unitary similitude} if, for some scalar $\beta$,            
\[
\langle h v, h v' \rangle    =  \beta     \langle   v', v \rangle, \,\, \quad \forall \,\, v, v' \in V. 
 \]    
 \end{enumerate}
 }
 \end{Def}        

We can now state the main technical result of \cite{RV}. 
In the form of the corollary below, this plays a crucial role in our adaptation of Tupan's method.
\begin{ThmA}
Let $g \in \guv$ with $\mu(g) = \beta$. Then there is an anti-unitary involution $h_1$  
and an anti-unitary similitude $h_2$ with $h_2^2 = \beta$ such that $g = h_1 h_2$.
\end{ThmA}

Once and for all we fix an anti-unitary involution $h \in {\rm Aut}_F(V)$ and set $\ui g = \mu(g)^{-1} hgh^{-1}$ for $g \in \guv$. 
Thus $\iota$ is a continuous automorphism of $\guv$ of order two. Further $\iota$ restricts to the automorphism  
$g \mapsto hgh^{-1}:\uv \to \uv$ 
which by obtuseness we again denote by $\iota$.

For $a \in \guv$, we set $\ut a = \ui a^{-1}$, so that $\theta$ (resp.~$\theta \, |_{ \, \uv})$ is an involutary anti-isomorphism of $\guv$ (resp.~$\uv$).  

\begin{Cor}
For each $a \in \guv$, there is an $x \in \uv$ with ${}^{\theta}x = x$ such that $xax^{-1} = {}^{\theta}a$.
\end{Cor}

\begin{proof}
Let $a \in \guv$ and put $\mu(a) = \beta$. 
By Theorem~A, we have $a = h_1 h_2$ for an anti-unitary involution $h_1$ and an anti-unitary similitude $h_2$ such that 
$h_2^2  = \beta$. Hence  $h_2^{-1} = \beta^{-1} h_2$ and 
\begin{align*} 
    {}^{\theta}a  &= \beta \, h \, h_2^{-1} h_1^{-1} h^{-1} \\
                        &= \beta \, h \, \beta^{-1} h_2 \, h_1^{-1} h^{-1} \\
                        &=  (h h_1) \,  (h_1 h_2) \,  (hh_1)^{-1} \\
                        &=   (h h_1) \, a \,  (hh_1)^{-1}.
                        \end{align*}
 Now $h h_1 \in \uv$ and 
 \begin{align*} 
 {}^{\theta} hh_1  &= h  (hh_1)^{-1} h^{-1} \\
                            &= h h_1 h  h^{-1} \\
                            &= h h_1, 
                            \end{align*} 
and thus we can take $x = h h_1$.                        
\end{proof}
 
Our main goal is to prove the following.  

\begin{ThmB}  
The maps $\iota:\uv \to \uv$ and $\iota:\guv \to \guv$ are dualizing involutions. 
\end{ThmB}

\section{Framework for proof of Theorem~B}
Again let $G$ be the group of $F$-points of a reductive $F$-group. 
Under the hypotheses in \S\ref{hypotheses} below, we show that the essential thread of Tupan's line of argument carries over to $G$. 
In particular, subject to these hypotheses, $G$ admits a dualizing involution. 

\subsection{}  \label{hypotheses}
Let $\theta:G \to G$ be an involutary anti-automorphism (i.e., $\theta(xy) = \theta(y) \theta(x)$ for all $x, y \in G$ and $\theta^2 = 1$). 
Writing $\lag$ for the Lie algebra of $G$, the differential  $d \theta:\lag \to \lag$ is an involutary anti-automorphism of $\lag$ 
which we often denote simply by $\theta$. 
In situations where we need to consider both maps, we sometimes 
write $\theta_G$ for the map on the group and  $\theta_\lag$ for the induced map on
the Lie algebra. As before, we set
\begin{equation} \label{iota}
\ui g = \ut g^{-1}, \quad  g \in G, 
\end{equation}
so that $\iota:G \to G$ is an involutary automorphism of $G$.  For $x \in G$, let ${\rm Int} (x)$ denote the inner automorphism 
$g \mapsto xgx^{-1}:G \to G$. As usual, we write ${\rm Ad} (x)$ for the induced automorphism of $\lag$, that is, 
${\rm Ad}(x) = d \, {\rm Int } (x) $ for $x \in G$. 

We impose the following hypotheses for the remainder of the section. 

\begin{hypoth}
$   $ 
\begin{enumerate}[{\bf (1)}]
\item
There is an  $\fo$-lattice $\ml \subset \lag$ and a map $c: \lag_1 \to G$ for $\lag_1 \subset \lag$ such that the following hold.
\begin{enumerate}[{\rm (a)}]
\smallskip 
\item
$\ut \lag_1 = \lag_1$ and $\theta_G \circ c = c \circ \theta_\lag$.
\smallskip 
\item
${\rm Ad} (x)  \lag_1 = \lag_1$  and 
${\rm Int} (x)  c(X)  = c ({\rm Ad} (x) X)$ for all $x \in G$ and $X \in \lag$.
\smallskip 
\item
$\ut \ml = \ml$ and $\varpi \ml \subset \lag_1$.
\smallskip 
\item
For each $k \geq 1$, the restriction $c \mid \varpi^k \ml$ is a homeomorphism onto a compact open subgroup of $G$. In particular,  
the family $\{c (\varpi^k \ml) \}_{k \geq 1}$ consists of compact open subgroups and forms a neighborhood basis of the identity in $G$. 
\end{enumerate}

\medskip

\item
For each $a \in G$, there is an $x \in G$ with ${}^{\theta}x = x$ such that $xax^{-1} = {}^{\theta}a$
\end{enumerate}
\end{hypoth}

The Corollary to Theorem~A shows that Hypothesis (2) holds for the classical and similitude groups $\uv$ and $\guv$. 
We will verify the various parts of Hypothesis (1) for these groups in \S\ref{verification}.

\begin{Rmk}
To obtain Tupan's setting \cite{Tupan}, we take $G = {\rm GL}_n(F)$ and $\lag = {\rm M}_n(F)$. The map $\theta$ is simply 
the transpose on $G$ and $\lag$. Further, $\ml = {\rm M}_n(\fo)$,  
$\lag_1 = \{ X \in {\rm M}_n(F) : \det(1+X) \neq 0 \}$ and $c:\lag_1 \to G$ is given by $c(X) = 1+X$. 
It is immediate that Hypothesis (1) holds. As noted in the introduction, that hypothesis (2) holds is a classical result in linear algebra. 
\end{Rmk}

It is convenient to introduce the following terminology. 
\begin{Def}
{\rm  A  subset $S$ of $G$ is}  \it conjugate-$\theta$-stable {\rm  if there is a $g \in G$ such that $\ut S = gSg^{-1}$. }
\end{Def}

We state a key technical result, our generalization of  \cite{Tupan} Theorem~1. 

\begin{Prop}
Any compact open subset of $G$ can be decomposed as a disjoint union of finitely many 
conjugate-$\theta$-stable compact open subsets. 
\end{Prop} 

\subsection{}
Granting the proposition, we show that it leads quickly to the main result. 

\begin{Thm}
The map $\iota:G \to G$ is a dualizing involution. 
\end{Thm}

Before the proof, we recall some background and set up some notation. 
Let $C_c^\infty(G)$ denote the space of locally constant compactly supported functions $f:G \to \cx$. 
Let $(\pi, V)$ be a smooth representation of $G$ and let $f \in C_c^\infty(G)$. The operator $\pi(f):V \to V$ is given by 
\[
\pi(f)v  = \int_G f(g) \pi(g) v \, dg, \quad \quad v \in V, 
\]
where the integral is with respect to a fixed Haar measure on $G$.  Suppose now that $(\pi,V)$ is irreducible. It follows that $(\pi,V)$ is 
{\it admissible}, that is, the space $V^K$ of $K$-fixed vectors has finite dimension for any open subgroup $K$ of $G$  \cite{Jac}.  Thus
the image of $\pi(f)$ has finite dimension and so $\pi(f)$ has a well-defined trace. The resulting linear functional 
$f \mapsto \tr \, \pi(f): C_c^\infty(G) \to \cx$ is the {\it distribution character} of $\pi$. It determines $\pi$ up to equivalence (\cite{BZ} 2.20).
With $\fc(g) = f(g^{-1})$ for $f \in C_c^\infty(G)$ and $g \in G$,  it is straightforward to check that $ \tr \, \pic(f) =  \tr \, \pi(\fc)$. 

\begin{proof}
For any compact open subset $C$ of $G$,  we write $\chi_C$ for the characteristic function of $C$. 
If ${}^{\theta}C = gCg^{-1}$ for $g \in G$ then 
\begin{equation} \label{C-conj}
  \pi(\chi_{{}^{\theta}C}) = \pi(g) \pi(\chi_C)\pi(g)^{-1}.
  \end{equation}
By Proposition~\ref{hypotheses}, any element of $C_c^\infty(G)$ can be written as a linear combination of characteristic functions of 
conjugate-$\theta$-stable compact open subsets of $G$. We set $\thf(g) = f({}^{\theta}g)$ for $f \in C_c^\infty(G)$ and $g \in G$. 
Using (\ref{C-conj}), it follows that  
\[
      \tr  \,  \pi(f)    =   \tr \, \pi(\thf), \quad \quad  \forall \,\, f \in C_c^\infty(G).
      \]
Now $\pii(f) = \pi( (f^{\vee}) {}^{\theta})$ and thus 
\begin{align*} 
           \tr \, \pii(f)  &=  \tr \, \pi( (f^{\vee}){}^{\theta})) \\  &= \tr \, \pi(\fc) \\ &= \tr \, \pic(f), \quad \quad  \forall \,\, f \in C_c^\infty(G).
       \end{align*}    
Therefore $\pii \cong \pic$.        
\end{proof}

\subsection{} \label{lemma1}
We now begin to work towards a proof of Proposition~\ref{hypotheses}.
For $x \in G$, let 
\[
      \ml(x) = {\rm Ad}(x^{-1})   \ml  \cap \ml.
\]

\begin{Lem} 
Suppose ${}^{\theta} x = x$ for $x \in G$. Then ${}^{\theta} \ml(x) = {\rm Ad}(x) \ml(x)$.
\end{Lem}

\begin{proof}
For any $x \in G$ and $X \in \lag$, 
\begin{equation}   \label{Ad-identity}
\ut ( {\rm Ad}(x) X)  = {\rm Ad}(\ut x^{-1} ) (\ut X). 
\end{equation} 
To check this, note that the left side is 
\begin{align*} 
        (d \theta_G \circ d \, {\rm Int}(x))( X )    &=  d(\theta_G  \circ {\rm Int}(x) )  ( X)   \\
                                                                     &=  d( {\rm Int}(\, {}^{\theta_G} x^{-1} )  \circ \theta_G ) ( X ) \\
                                                                     &=   {\rm Ad} ({}^{\theta_G} x^{-1}) ( {}^{\theta_\lag} X).
\end{align*}

Thus 
\begin{align*} 
  \ut \ml(x)   &= \ut ( {\rm Ad}(x^{-1})\ml  ) \cap \ut \ml  \\
                      &=  {\rm Ad} (\ut x)   ( \ut \ml )  \cap \ut \ml   \quad \text{(by (\ref{Ad-identity}))} \\
                   &=  {\rm Ad} (\ut x)   \ml  \cap  \ml \quad \quad \text{(using $\ut \ml = \ml$)} \\
                   &=  {\rm Ad} (\ut x) \, ( \ml \cap {\rm Ad}(\ut x^{-1}) \ml )  \\
                   &=  {\rm Ad} (\ut x) \ml (\ut x).
\end{align*}
In particular, if $\ut x  = x$ then ${}^{\theta} \ml(x) = {\rm Ad}(x) \ml(x)$.
\end{proof}

\subsection{} \label{lemma2}
Our next observation 
is central to  the proof of Proposition~\ref{hypotheses}. 
\begin{Lem}
Let $a, x \in G$ with ${}^{\theta}x= x$ and $xax^{-1} = {}^{\theta}a$.  

\begin{enumerate}[{\rm (1)}]
\item
The set $c(\varpi^k \ml(x) )$ is a compact open subgroup of $G$ (for $k \geq 1$).  

\medskip

\item
The set $a c(\varpi^k \ml(x) )$ is a conjugate-$\theta$-stable compact open neighborhood of $a$ (for $k \geq 1$).  
\end{enumerate}

\end{Lem}

\begin{proof}
By hypotheses (1)(b) and (1)(d), 
\[
   c ( \varpi^k \ml(x) )   =   x^{-1} c ( \varpi^k \ml) x \cap c(\varpi^k \ml). 
   \] 
This proves part~(1). To prove part~(2), note that 
\begin{align*} 
    {}^{\theta} (a c(\varpi^k \ml(x) )   &=  {}^{\theta}c(\varpi^k  \ml(x) )  \, {}^{\theta}a \\
                                              &=   c (\varpi^k \cdot {}^{\theta}\ml(x) ) \, xax^{-1} \hspace{20pt} \text{(by hypothesis~(1)(a))} \\
                                              &=  x   c(\varpi^k \, \ml(x)) x^{-1} \, xax^{-1} \hspace{7pt} \text{(by hypothesis~(1)(b) and Lemma \ref{lemma1})}\\
                                                       &=  ( xa^{-1} ) \,a c(\varpi^k \ml(x) )\, (xa^{-1})^{-1}, 
                                                       \end{align*}                                                        
so that $a c(\varpi^k \ml(x))$ is conjugate-$\theta$-stable.  \end{proof} 

\subsection{}
Let $K_0$ be any compact open subgroup of $G$. The group $K_0$ acts on $\lag$ via the (restriction of the) adjoint action. We set 
\[
   K = {\rm Stab}_{K_0} \, \ml  \subset K_0.
\]
Thus $K$ is a compact open subgroup of $G$ that stabilizes $\ml$. The coset space $K \backslash G$ is countable and 
so the collection of cosets that contain some $\theta$-fixed element is also countable. We label these cosets as 
$\{ K d_i : i \geq 1\}$ where ${}^{\theta} d_i = d_i$ (i.e., each representative $d_i$ is $\theta$-fixed). 

Note that if ${}^{\theta}x= x$ and $x \in Kd_i$, then 
\begin{equation} \label{eqn}
       \ml (x) = \ml(d_i). 
       \end{equation}
 Indeed, if $x = kd_i$ with $k \in K$, then 
 \begin{align*} 
     \ml(x)    &=   {\rm Ad} (d_i^{-1}) {\rm Ad}(k^{-1}) \ml   \cap \ml \\
                  &=   {\rm Ad} (d_i^{-1})  \ml  \cap \ml \\
                  &=    \ml(d_i).
                  \end{align*}

\subsection{} 
We can now prove Proposition~\ref{hypotheses} which for convenience we restate as follows. 

\begin{Prop}
For any  compact open subset $C$ of $G$, there exist finitely many conjugate-$\theta$-stable compact open subsets 
$C_1, \ldots, C_s$ of $C$ such that $C = \bigsqcup_{i=1}^s C_i$. 
\end{Prop}
\begin{proof} 
By hypothesis, the family $\{ c(\varpi^k \ml) \}_{k \geq 1}$ is a neighborhood basis of $1 \in G$
consisting of compact open subgroups. 
It follows that it suffices to prove the result for $C = b c(\varpi^{l_0} \ml)$ for any  $b \in G$ and any $l_0 \geq 1$ 

For each $a \in C$, we choose an $x_a \in G$ with ${}^{\theta}x_a = x_a$ such that $x_{\mathstrut a} a x_a^{-1} = {}^{\theta}a$. 
In the special case ${}^{\theta} a =a$, we always take $x_a = 1$.  
For each $i \geq 1$, let
\[
   \mathcal{C}_i = \{ a \in C :   x_a \in Kd_i \}, 
   \]
so that $C = \bigsqcup_{i \geq 1} \mathcal{C}_i$. 

Assume first that there is some $\theta$-fixed element $a$ in $C$. Then $x_a = 1$ and 
$\ml (x_a) = \ml$, so that 
\[
     C = a c(\varpi^{l_0} \ml) = a c(\varpi^{l_0} \ml(x_a) ). 
     \]
Thus $C$ is itself conjugate-$\theta$-stable by Lemma~\ref{lemma2}~(2). 

Suppose now that $C$ contains no $\theta$-fixed element. In this case we use the following inductive construction. 
We first choose $l_1 \geq 1$ such that 
$c (\varpi^{l_1} \ml (d_1) ) \subset   c (\varpi^{l_0} \ml)$. 
By induction, for $i \geq 2$  we can choose $l_i \geq 1$ such that 
\[
c (\varpi^{l_i}   \ml(d_i) ) \subset c (\varpi^{l_{i-1}}   \ml(d_{i-1}) ) \subset \cdots \subset 
                                            c (\varpi^{l_1}  \ml(d_1) ) \subset c (\varpi^{l_0}   \ml).    
\]
Now if $a \in C_i$, then $a$ belongs to the open set $a c(\varpi^{l_i} \ml (d_i) ) \subset C = a c(\varpi^{l_0} \ml )$. Further, 
by (\ref{eqn}), $a c(\varpi^{l_i} \ml (d_i) ) = a c(\varpi^{l_i} \ml (x_a) )$ and thus $a c(\varpi^{l_i} \ml (d_i) )$ is conjugate-$\theta$-stable by 
Lemma~\ref{lemma2}~(2). 
In this way, we associate a conjugate-$\theta$-stable open neighborhood of $a$ to each $a \in C$.  By construction, any two of these neighborhoods are either disjoint
or nested (i.e., one is contained in the other).  As $C$ is compact, it can be covered by finitely many such neighborhoods. 
Taking the maximal elements (with respect to inclusion) in any such cover, we  obtain the desired
decomposition of $C$.  
\end{proof}

\section{Cayley maps}   \label{Cayley maps}
We introduce a Cayley map  for similitude groups and collect some of its properties. 
We use these and corresponding properties of the classical Cayley map to verify in  
\S\ref{verification} that the groups $\uv$ and $\guv$ satisfy  Hypothesis~(1) of \S\ref{hypotheses}. 
In this way, Cayley maps underpin our proof of Theorem~B. 
As noted above, our use of such maps means that we have to exclude the case of even residual characteristic. We note also that 
there are more refined treatments of the classical Cayley map in the literature. For example, 
Lemma \ref{Cayley-compact-open} (2) below is a special case of  \cite{Morris} Theorem 2.13 (c).  

\subsection{}
Let  $a \in \ev$. By non-degeneracy of $\vform$, there is a unique $a^\ast \in \ev$ such that 
\[
         \langle a v, v' \rangle =  \langle v, a^\ast v' \rangle,  \,\,\, \forall \,\, v, v' \in V.  
\]
The resulting map $a \mapsto a^\ast: \ev  \to  \ev$ is a $\tau$-linear anti-involution. Explicitly, for all $\lambda \in E$ and
$a, b \in \ev$, 
\begin{enumerate}[(1)]
\item
$(\lambda a)^\ast = \tau(\lambda) a^\ast$ and $(a+b)^\ast = a^\ast + b^\ast$;
\item
$(a^\ast)^\ast = a$ and $(ab)^\ast = b^\ast a^\ast$. 
\end{enumerate}
We use these properties without comment below. Note      
$a \in \uv$ if and only if $a \,a^\ast = 1$. Similarly $a \in \guv$ if and only if $a \,a^\ast = \beta$ for some scalar $\beta$
in which case $\mu(a) = \beta$. 

We write $\lieg$ and $\lietg$ for the Lie algebras of $\uv$ and $\guv$ respectively. Thus 
\begin{align*}
   &\lieg = \{ X \in \ev :  \langle X v, v' \rangle  + \langle  v, X v' \rangle  = 0,   \, \forall \,\, v, v' \in V\},  \\
   \lietg = \{ &X \in \ev :  \langle X v, v' \rangle  + \langle  v, X v' \rangle  = \alpha \, \langle v, v' \rangle,  
            \,\, \text{for some scalar $\alpha = \alpha(X)$}, 
                                                        \, \forall \,\, v, v' \in V\}.  
\end{align*}
That is, $X \in \lieg$ if and only if $X+X^\ast = 0$  and $X \in \lietg$ if and only if $X + X^\ast = \alpha(X)$.

\subsection{}
Consider the dense open subset of $\lietg$ given by 
\begin{equation} \label{domain}
\lietg^1 = \{ X \in \lietg : 1 + \alpha(X) \neq 0, \, \det(1+X) \neq 0 \}. 
\end{equation}
The similitude Cayley map $c$ is defined by 
\begin{equation} \label{sim-cayley}
X \overset{c}{\longmapsto} \left( 1 - \dfrac{X}{1+\alpha} \right) (1+X)^{-1}: \lietg^1 \to \guv
       \end{equation}
where $\alpha = \alpha(X)$.

To see that $c(X) \in \guv$, note that 
\begin{align*} 
    c(X) c(X) ^\ast &=  \left( 1 - \dfrac{X}{1+\alpha}\right) (1+X)^{-1} \left( 1 - \dfrac{X^\ast}{1+\alpha}\right) (1+X^\ast)^{-1}  \\
                            &=  \left( 1 - \dfrac{X}{1+\alpha}\right) (1+X)^{-1} \left( 1 - \dfrac{\alpha - X}{1+\alpha}\right) (1+\alpha - X)^{-1}  \\
                            &= \left( 1 - \dfrac{X}{1+\alpha}\right) (1+X)^{-1} ( 1 + X) \left(1- \dfrac{X}{1+\alpha}\right)^{-1} (1+\alpha)^{-2}   \\
                            &=   (1+\alpha)^{-2}.   
                            \end{align*} 
Thus $c(X)$ indeed belongs to $\guv$ and 
\begin{equation} \label{c(X)-multiplier}  
 \mu( c(X) ) = \dfrac{1}{(1+\alpha)^2}.
 \end{equation}

Let  
\[
\lieg^1 = \lietg^1 \cap \lieg = \{ X \in \lieg :  \det(1+X) \neq 0 \}.
\] 
Then (\ref{sim-cayley}) restricts to the classical Cayley map 
\[
X \overset{c}{\longmapsto} ( 1 - X) (1+X)^{-1}: \lieg^1 \to \uv.
\]

\subsection{}   \label{Cayley-fibers}
In contrast to its classical version, the similitude Cayley map is not injective. For later use, we describe its image and fibers (in more detail than we strictly require).

For $g \in  \guv$, we write $\mu = \mu(g)$.  For $X \in \lietg^1$, we set $\lambda = \dfrac{1}{1+\alpha}$ where $\alpha = \alpha(X)$, so that 
 \[
 c(X) = \dfrac{1-\lambda X}{1+X}.
 \]
For $\lambda \in F$ such that $\lambda + g$ is invertible, we also set 
\[
  X_\lambda = X_\lambda(g) =  \dfrac{1-g}{\lambda +g}.
  \]
Note that solving $c(X) = g$ for $X$ gives $X = X_\lambda$ provided $\lambda + g$ is invertible.  

\begin{Prop}
The image of the similitude  Cayley map {\rm (\ref{sim-cayley})} consists of all $g \in \guv$ such that 
{\rm (a)}
$\mu =  \lambda^2$ for some $\lambda \in F^\times$ and
{\rm (b)}
either $\mu \neq 1$ and at least one of $\pm \lambda +g$ is invertible or $\mu = 1$ and $1 + g$ is invertible.

Let $g \in \guv$ belong to this image with $\mu = \lambda^2$.   
\begin{enumerate}[{\rm(1)}]
\item
If $\mu \neq 1$ and $\lambda + g$ is invertible but not $-\lambda+g$, then $X_\lambda(g)$ is the unique preimage of $g$. 

\item
If $\mu \neq 1$ with $\pm \lambda + g$ both invertible, then $g$ has precisely two preimages, 
$X_\lambda(g)$ and $X_{-\lambda}(g)$.

\item
If $\mu = 1$ and $g \neq 1$ with $1+g$ invertible, then $X_1(g)$ is the unique preimage of $g$.

\item
The preimage of $1\in \guv$ is infinite. It consists of $0$ and the elements $X \in \lietg^1$ such that $\alpha(X) = -2$. 

\end{enumerate}

\end{Prop} 

\begin{proof}
Suppose $g = c(X)$ for some $X \in \lietg^1$. We have $\mu = \lambda^2$ by (\ref{c(X)-multiplier}).  
Assume first that $\lambda \neq -1$, equivalently $\alpha \neq -2$. Then 
\begin{align*} 
 \lambda + g  &= \lambda + \dfrac{1-\lambda X}{1+X} \\
                      &= \dfrac{\lambda +1}{1+X} 
                      \end{align*}
and thus $\lambda +g$ is invertible. 
In the case $\lambda = -1$, we have $c(X) = 1$. It follows that the image of the similitude Cayley map is as stated. 

We can reverse this reasoning to determine the map's fibers. 
Indeed, suppose $g \in \guv$ and $\lambda^2 = \mu$. 
If $\lambda +g$ is invertible, then 
\begin{align*} 
   1 + X_\lambda  &=  1 + \dfrac{1-g}{\lambda + g} \\
                            &=   \dfrac{\lambda + 1}{\lambda +g}. 
                            \end{align*}
Hence $1+X_\lambda$ is invertible if and only if $\lambda \neq -1$. Moreover, using $g g^\ast = \lambda^2$,  
\begin{align*} 
  X_\lambda + X_\lambda^\ast  &= \dfrac{1-g}{\lambda + g}  + \dfrac{1-g^\ast}{\lambda + g^\ast}  \\
                                               &=  \dfrac{1-g}{\lambda + g}  +  \dfrac{1- \lambda^2 g^{-1} }{\lambda + \lambda^2 g^{-1}}  \\
                                               &=  \dfrac{1-g}{\lambda + g}  +  \dfrac{(\lambda^{-1}g - \lambda) \lambda g^{-1}}{(\lambda + g) \lambda g^{-1}}  \\
                                               &=  \dfrac{ 1-g + \lambda^{-1}g - \lambda}{\lambda +g} \\
                                               &=  \dfrac{(\lambda +g) (\lambda^{-1} - 1)}{\lambda + g} \\
                                               &= \lambda^{-1} - 1. 
                                                  \end{align*}
Thus $X_\lambda \in \lietg^1$ provided $\lambda \neq -1$ and $c(X_\lambda) = g$ (by construction of $X_\lambda$). 
Statements (1) through (4) all follow. 
\end{proof}

\subsection{}  \label{lattice-values}
Let $L$ be an $\eo$-lattice in $V$. 
(That is, $L$ is an $\eo$-submodule of $V$ such that $L \otimes_{\eo} E = V$.  Equivalently, $L$ is a compact open $\eo$-submodule of $V$.)
 For later purposes, we  assume also that $h(L) = L$ where $h$ is our fixed anti-unitary involution.  It is immediate that such lattices exist. Indeed, for any $\eo$-lattice $L_0$ in $V$, we may take  $L = L_0 \cap h(L_0)$.
We set 
 \[
 \mlh = \{ a \in \ev :  a(L) \subset L \}.
 \]
This is an $\eo$-order in $\ev$ (i.e., an $\eo$-lattice in $\ev$ that is also a subring) and hence also an $\fo$-order in $\ev$. We put 
 \begin{equation} \label{lie-alg-lattices}
  \mtl =  \mlh \cap \lietg, \quad \mttl = \mlh \cap \lieg.
  \end{equation}
It follows that $\mtl$ and $\mttl$ are $\fo$-lattices in $\lietg$ and $\lieg$ respectively.  


\begin{Lem}
Let $X \in \mtl$ and $g \in ( \, 1+ \varpi^k \mlh \,  ) \cap \guv$ for $k \geq 1$. Then 
\begin{enumerate}[{\rm (1)}]
\item
$\alpha(X) \in \fo$,
\item
$\mu(g) \in 1 + \fp^k$. 
\end{enumerate}
\end{Lem} 

\begin{proof}
Let $\fl$ denote the fractional ideal of $F$ generated by the elements $\langle u, v \rangle$ for $u, v \in L$.  
For any such $u$ and $v$, we have $ \langle X u, v \rangle + \langle u, X v \rangle \in \fl$ where $\alpha = \alpha(X)$, whence  
$\alpha \langle u, v \rangle \in \fl$. It follows that $\alpha \fl \subset \fl$ and so $\alpha \in \fo$. 

Write $g = 1+ \varpi^k X$ for $X \in \mlh$. 
With $\beta = \mu(g)$, we then have 
\[
     \langle (1+\varpi^k X) u, (1+\varpi^k X) v \rangle = \beta \, \langle u, v \rangle, \quad \forall \,\, u, v \in L.
     \]
Expanding and rearranging gives
\[
\varpi^k \langle X u , v \rangle + \varpi^k \langle u, X v \rangle  = ( \beta - 1 ) \,  \langle u, v \rangle,  \quad \forall \,\, u, v \in L.
\]
Thus $ ( \beta - 1) \fl \subset \fp^k \fl$ and $\beta - 1 \in \fp^k$. 
\end{proof}

\subsection{}    \label{Cayley-compact-open}
 The family of compact open subgroups $\{ 1+ \varpi^k \mlh \}_{k \geq 1}$ is a neighborhood basis of the identity in $\av$. 
 Thus 
 $\{ \,( 1+ \varpi^k \mlh \,)  \cap \guv ) \}_{k \geq 1}$ 
 and 
  $\{ \,( 1+ \varpi^k \mlh \,)  \cap \uv ) \}_{k \geq 1}$ 
 form neighborhood bases of the identity in $\guv$ and $\uv$ (respectively) that consist again of compact open subgroups. 
 These families have a simple description in terms of the Cayley map. 
  
\begin{Lem} 
For any integer $k \geq 1$, 
\begin{enumerate}[{\rm (1)}]
\item
 $c(\varpi^k \mtl) =         ( \, 1+ \varpi^k \mlh \,  ) \cap \guv$, 
\item
$c(\varpi^k \mttl)      = ( \, 1+ \varpi^k \mlh \,  ) \cap \uv$.  
\end{enumerate}
Moreover, each restriction $c \mid \varpi^k \mtl$ and $c  \mid \varpi^k \mttl$ is a homeomorphism onto its image.  
\end{Lem}

\begin{proof}
Suppose $X \in \varpi^k \mtl$. By Lemma~\ref{lattice-values}~(1),
 $\alpha = \alpha(X) \in \fp^k$ and thus $(1+\alpha)^{-1}  \in 1+\fp^k$. 
It follows that 
\[
c(X) = \left( 1 - \dfrac{X}{1+\alpha}\right) (1+X)^{-1}   \,\, \in \,\, 1+\varpi^k \mlh.
\]
Hence $c(\varpi^k \mtl) \subset ( \, 1+ \varpi^k \mlh \,  ) \cap \guv$. 
Taking $\alpha = 0$, we see also that $c(\varpi^k \mttl) \subset ( \, 1+ \varpi^k \mlh \,  ) \cap \uv$.

To prove the opposite containments, let $g \in  ( \, 1+ \varpi^k \mlh \,  ) \cap \guv$ and set $\mu = \mu(g)$. 
Then $\mu \in 1+\fp^k$ by Lemma~\ref{lattice-values}~(2).
Since the residual 
characteristic is odd, there is a unique $\lambda \in 1+\fp^k$ such that $\mu = \lambda^2$.  Writing
$g= 1+\varpi^k X$ for $X \in \mlh$, we have 
\begin{align*}
\lambda + g  &= 1 + \lambda + \varpi^k X   \\  &=  (1+ \lambda) \left( 1 + \dfrac{\varpi^k X}{1+\lambda} \right).
\end{align*}
Note $1+\lambda \in \fo^\times$ (again since the residual characteristic is odd). 
Thus $1 + \dfrac{\varpi^k X}{1+\lambda}  \in  1 + \varpi^k \mlh$. In particular,  $\lambda + g$ is invertible. 
By Proposition~\ref{Cayley-fibers}, $c(X_\lambda) = g$ where 
\[
     X_\lambda  = -     \dfrac{\varpi^k X}{1+\lambda} \,  \left( 1 + \dfrac{\varpi^k X}{1+\lambda} \right)^{-1}. 
\]
It follows that $X_\lambda \in \varpi^k \mtl$. This proves~(1). To complete the proof of (2),  we have only to note that $\mu = 1$ implies
$\lambda = 1$ in which case $X_1 \in  \varpi^k \mtl \cap \lieg = \varpi^k \mttl$. 

Finally, each restriction $c \mid \varpi^k \mtl$ and $c  \mid \varpi^k \mttl$ 
is a continuous map on a compact space. Further, by Proposition~\ref{Cayley-fibers},  each map is a bijection, and thus
each is a homeomorphism.    
\end{proof}


\section{Proof of Theorem~B} \label{verification}
Recall that $h \in {\rm Aut}_F(V)$ is an  anti-unitary involution and that $\ui g = \mu(g)^{-1} hgh^{-1}$ for $g \in \guv$. 
Thus 
\[
\ut g = \mu(g) h g^{-1} h^{-1}, \quad g \in \guv.
\]
The differential of $\theta$ is an involutary anti-automorphism of $\lietg$ which for simplicity we usually denote by the same symbol. 
It is given explicitly by 
\begin{equation}   \label{explicit-theta}
 {}^{\theta} X =  \alpha(X) - hXh^{-1}, \quad \,\, X \in \lietg. 
\end{equation} 

The map $\iota$ restricts to an automorphism of $\uv$ which we again denote by $\iota$.
Thus $\ui g = hgh^{-1}$ for $g \in \uv$. 

We restate Theorem~B. 

\begin{Thm}
The maps   $\iota:\uv \to \uv$ and  $\iota:\guv \to \guv$ are dualizing involutions. 
\end{Thm}

To complete the proof, we only have to check Hypothesis~(1) in \S\ref{hypotheses}. 
 
\subsection{}

We first restate the four parts of Hypothesis~(1) and then verify each part in turn for the similitude groups.
Let $G = \guv$ and $\lag = \lietg$. 
\vspace{8pt}

\begin{enumerate}[{\bf (1)}]
\item
There is an  $\fo$-lattice $\ml \subset \lag$ and a map $c: \lag_1 \to G$ for $\lag_1 \subset \lag$ such that the following hold.
\begin{enumerate}[{\rm (a)}]
\item
$\ut \lag_1 = \lag_1$ and $\theta_G \circ c = c \circ \theta_\lag$.
\smallskip 
\item
${\rm Ad} (x)  \lag_1 = \lag_1$  and 
${\rm Int} (x)  c(X)  = c ({\rm Ad} (x) X)$ for all $x \in G$ and $X \in \lag$.
\smallskip 
\item
$\ut \ml = \ml$ and $\varpi \ml \subset \lag_1$;
\smallskip 
\item
For each $k \geq 1$, the restriction $c \mid \varpi^k \ml$ is a homeomorphism onto a compact open subgroup of $G$. In particular,  
the family $\{c (\varpi^k \ml) \}_{k \geq 1}$ consists of compact open subgroups and forms a neighborhood basis of the identity in $G$. 
\end{enumerate}
\end{enumerate}

We put
\[
\lag_1 = \{ X \in \lag: 1+ \alpha(X)  \neq 0, \,\,\det(1+X) \neq 0, \,\, \det( 1 +\alpha(X)  - X) \neq 0 \}.
\]
Note that $\lag_1$ is contained in the domain~(\ref{domain}) of the similitude Cayley map~(\ref{sim-cayley}). 
We take $c:\lag_1 \to G$ to be the restriction of this map to $\lag_1$ and put  $\ml = \mtl$ (see  (\ref{lie-alg-lattices})).
 
 
\vskip4pt
\noindent \underline{(a)} \quad
To show that $\ut \lag_1 = \lag_1$, it suffices to prove $\ut X \in \lag_1$ for $X \in \lag_1$.  

To this end, we check first that 
\begin{equation} \label{alpha-theta}
  \alpha(\ut X) = \alpha(X), \quad X \in \lag.
\end{equation}
We have $X^\ast = \alpha - X$ with $\alpha = \alpha (X)$.  Hence 
\begin{align*}
(\ut X)^\ast      &=  (\alpha - hXh^{-1})^\ast \\
                        &= \alpha  -  (hXh^{-1})^\ast. 
                        \end{align*}
                        Using that $h$ is anti-unitary with $h^2 = 1$, a quick calculation shows that $(hXh^{-1})^\ast = h X^\ast h^{-1}$. Thus 
\begin{align*}
(\ut X)^\ast      &=  \alpha - h X^\ast h^{-1} \\
                        &=  \alpha - h (\alpha - X)h^{-1} \\
                        &= h Xh^{-1} \\
                        &= \alpha - \ut X
                        \end{align*} 
which proves (\ref{alpha-theta}).

For any $X \in \lag$, 
\begin{align*} 
   \det (1 + \ut X)   &=  \det (1 + \alpha - h Xh^{-1} ) \\
                             &=  \det h (1+ \alpha - X) h^{-1} \\
                             &= \tau ( \det (1 + \alpha - X) ).
                             \end{align*}
 Similarly, 
 \begin{align*} 
 \det (1 + \alpha - \ut X ) &= \det ( 1 + h X h^{-1} ) \\
                            &=  \det h (1+  X) h^{-1} \\
                             &= \tau ( \det (1 + X) ).
                             \end{align*}               
Thus $\ut X \in \lag_1$ for $X \in \lag_1$ and  $\ut \lag_1 =  \lag_1$. 

Further for any $X \in \lag_1$, 
\begin{align*} 
    \ut c(X)     &=  \dfrac{1}{(1+\alpha)^2}  \,  h  \,  c(X)^{-1} h^{-1} \\
                    &=  \dfrac{1}{(1+\alpha)^2} \,   \left( \, 1 - \dfrac{hXh^{-1}}{1+\alpha}\,\right)^{-1} (1+hXh^{-1})  \\
                    &=  \dfrac{1}{(1+\alpha)^2}   \,  \left( \dfrac{1 + \alpha - hXh^{-1}}{1+\alpha} \right)^{-1} (1+hXh^{-1})  \\
                    &=  \dfrac{1}{(1+\alpha)}  \,    ( 1 + \ut X )^{-1} (1+hXh^{-1}) \\
                    &=  ( 1 + \ut X )^{-1} \, \left(\, \dfrac{1+\alpha - (\alpha - hXh^{-1})}{1+\alpha}\,\right) \\
                    &=  ( 1 + \ut X )^{-1} \, \left(1 - \dfrac{\ut X}{1+\alpha}\right) \\
                    &=   c ( \ut X).  
                     \end{align*}

\vskip4pt
\noindent \underline{(b)} \quad Let $x \in G$ with $x x^\ast = \beta$ and $X \in \lag$. We have ${\rm Ad}(x) X = x X x^{-1}$ and 
\begin{align*} 
         (x X x^{-1})^\ast    &=  (x^{-1})^\ast X^\ast x^\ast  \\
                                      &=  \beta^{-1} x   X^\ast  \beta x^{-1} \\
                                      &=      x X^\ast x^{-1} \\
                                      &=    x (\alpha - X) x^{-1} \\
                                      &=     \alpha - xXx^{-1}.
                                      \end{align*}
Thus  $\alpha({\rm Ad}(x) X) = \alpha(X)$.  It follows easily that ${\rm Ad}(x) \lag_1 = \lag_1$ and that ${\rm Int}(x) c(X) = c({\rm Ad}(x) X)$.

\vskip4pt
\noindent \underline{(c)} \quad We show first that $\ut \ml \subset \ml$ which implies $\ut \ml = \ml$. 

Let $X \in \ml$. By Lemma~\ref{lattice-values}~(1), $\alpha = \alpha(X) \in \fo$.  
As the $\eo$-lattice $L \subset V$ was chosen so that $h(L) = L$, it follows that $\ut X = \alpha - h Xh^{-1}$ preserves $L$ 
and so $\ut X \in \ml$. 

We have $\det (1 + \varpi X) \in 1 + \fp$. As $\alpha(\varpi X) \in \fp$, it follows also  also that $\alpha(\varpi X) \neq -1$ and 
$\det (1 + \alpha(\varpi X) - \varpi X) \in 1+ \fp$. In particular, $\varpi \ml \subset \lag_1$.

\vskip4pt
\noindent \underline{(d)} \quad This follows immediately from Lemma~\ref{Cayley-compact-open}.

\subsection{}
Next we verify Hypothesis~(1) for the classical groups $\uv$. 
We continue with the notation of the preceding subsection. In particular, 
$G = \guv$ and $\lag = \lietg$. We set  $G' = \uv$ and $\lag' = \lieg$.
We also put $\lag_1' = \lag_1 \cap \lag'$ and $\ml' = \ml \cap \lag'$
so that $\ml' = \mttl$ in the notation of (\ref{lie-alg-lattices}). We take $c:\lag_1' \to G'$ to be the classical Cayley map restricted to $\lag_1'$.
We have $\ut g = hg^{-1}h^{-1}$ for $g \in G'$. The induced map on $\lag'$ which we again denote by $\theta$ is 
\[
       \ut X = - h X h^{-1}, \quad X \in \lag'.
       \]

We need to show that (a)-(d) hold in this (primed) setting. For (a)-(c), the verifications of the preceding subsection all go through using 
$\alpha(X)=0$ for $X \in \lag'$.   Finally, for each $k \geq 1$, Lemma~\ref{Cayley-compact-open} gives that $c \mid \varpi^k \ml'$ is a homeomorphism onto a compact open subgroup of $G'$ and so (d) holds. 


\end{document}